\documentclass[11pt,leqno,draft]{article}
\usepackage[margin=2.5cm]{geometry}
\usepackage{amssymb, amsmath, amsthm}
\usepackage{enumerate}
\usepackage{color}
\usepackage[final]{graphicx}
\usepackage[final]{hyperref}
\usepackage[sort&compress,numbers]{natbib}
\usepackage{microtype}
 
\hypersetup{
    bookmarks=true,			
    unicode=false,			
    pdftoolbar=true,			
    pdfmenubar=true,			
    pdffitwindow=true,		
    pdftitle={My title},			
    pdfauthor={Author},		
    pdfsubject={Subject},		
    pdfnewwindow=true,		
    pdfkeywords={keywords},	
    colorlinks=true,			
    linkcolor=blue,			
    citecolor=blue,			
    filecolor=blue,			
    urlcolor=blue				
}

\newtheorem{theorem}{Theorem}
\newtheorem{lemma}{Lemma}

\newtheorem{proposition}[theorem]{Proposition}

\theoremstyle{definition}

\theoremstyle{remark}
\newtheorem*{remark}{Remark}

\newcommand{\eps}		{\varepsilon}
\newcommand{\R}		{\mathbb{R}}

\newcommand{\Z}		{\mathbb{Z}}
\newcommand{\F}		{\mathcal{F}}
\newcommand{\jj}		{\vec\jmath}

\newcommand{\norm}[1]	{\lVert#1\rVert}
\newcommand{\flr}[1]	{\left\lfloor #1 \right\rfloor}
\newcommand{\ceil}[1]	{\left\lceil #1 \right\rceil}
\renewcommand{\exp}[1]	{\operatorname{exp}\left( #1 \right)}

\renewcommand{\Pr}[2][]	{\mathbf{P}_{#1}\left\{ #2 \right\}}
\newcommand{\Ex}[1]	{\mathbf{E}\left\{ #1 \right\}}
\newcommand{\Exc}[1]	{\mathbf{E}\{ #1 \}}
\newcommand{\ind}[1]	{\mathbf{1}_{#1}}

\newcommand{\wh}		{\widehat}
\newcommand{\wt}		{\widetilde}

\DeclareMathOperator{\Vol}{Vol}
\DeclareMathOperator{\Bin}{Binomial}

\title{\textbf{Connectivity threshold of Bluetooth graphs}}

\author{
Nicolas Broutin \\
INRIA  
\and Luc Devroye \\
McGill University 
\thanks{
The second author's research was sponsored by NSERC Grant A3456 and FQRNT Grant 90-ER-0291.}
\and Nicolas Fraiman \\
McGill University 
\and G\'abor Lugosi \\
ICREA and Pompeu Fabra University
\thanks{
The fourth author acknowledges support by the Spanish Ministry of Science and Technology grant MTM2009-09063 and by the PASCAL Network of Excellence under EC grant no.\ 506778.}
}

\date{\today}
                                  
\begin{document}

\maketitle

\begin{abstract}
We study the connectivity properties of random Bluetooth graphs that
model certain ``ad hoc'' wireless networks. The graphs are obtained as
``irrigation subgraphs'' of the well-known random geometric graph
model. There are two parameters that control the model: the radius $r$
that determines the ``visible neighbors'' of each node and the number
of edges $c$ that each node is allowed to send to these. The
randomness comes from the underlying distribution of data points in
space and from the choices of each vertex. We prove that no
connectivity can take place with high probability for a range of
parameters $r, c$ and completely characterize the connectivity
threshold (in $c$) for values of $r$ close the critical value for connectivity
in the underlying random geometric graph. 
\end{abstract}

\section{Introduction}\label{sec:intro}

It is sometimes necessary to \emph{sparsify} a network: given a connected graph, one wants to extract a sparser yet connected subgraph. In general, the protocol should be distributed, in that it should not involve any global optimization or coordination for obvious scaling reasons. The problem arises for instance in the formation of Bluetooth ad-hoc or sensor networks \cite{WhHoCh2005a}, but also in settings related to information dissemination (broadcast or rumour spreading) \cite{BGPS06,DM10}. 

In this paper, we consider the following simple and distributed algorithm for graph sparsification. Let $G_n=(V,E)$ be a finite undirected graph on $|V|=n$ vertices and edge set $E$. A \emph{random irrigation subgraph} $S_n=(V,\wh{E})$ of $G_n$ is obtained as follows: Let $2\leq c_n < n$ be a positive integer. For every vertex $v\in V$, we pick randomly and independently, without replacement, $c_n$ edges from $E$, each adjacent to $v$. These edges form the set of edges $\wh{E} \subset E$ of the graph $S_n$ (if the degree of $v$ in $G_n$ is less than $c_n$, all edges adjacent to $v$ belong to $\wh{E}$). The main question is how large $c_n$ needs to be so that the graph $S_n$ is connected, with high probability. Naturally, the answer depends on what the underlying graph $G_n$ is. 

When $G_n = K_n$ is the complete graph then for constant $c_n = c$,
\citet{FF82} proved that $S_n$ is $c$-connected with high
probability. This model is also known as the random $c$-out graph. In
a subsequent paper, \citet{FF83} considered probability of the existence
of a Hamiltonian cycle. They showed that there exists $c\leq 23$ such
that a Hamiltonian cycle exists with probability tending to $1$ as $n$
tends to infinity. In a recent article \citet{BF09} proved that $c = 3$
suffices.

Apart from the complete graph, the most extensively studied case is
when $G_n=G_n(r_n)$ is a \emph{random geometric graph} defined as
follows: Let $X_1, \dots ,X_n$ be independent, uniformly distributed
random points in the unit cube $[0,1]^d$. The set of vertices of the
graph $G_n(r_n)$ is $V = \{1,\dots, n\}$ while two points $i$ and $j$
are connected by an edge if and only if the Euclidean distance between
$X_i$ and $X_j$ does not exceed a positive parameter $r_n$, i.e., $E =
\{(i,j): \norm{X_i - X_j} < r_n\}$ where $\norm{\cdot}$ denotes the
Euclidean norm. Many properties of $G_n(r_n)$ are well understood.  We
refer to the monograph of \citet{Pen03} for a survey. The graph $S_n=S_n(r_n,c_n)$ was introduced in the context 
of the Bluetooth network \cite{WhHoCh2005a}, and
is sometimes called the \emph{Bluetooth} or \emph{scatternet graph} with
parameters $n,r_n$, and $c_n$. The model was introduced and studied in
\cite{FMPP04, PR04, DHMPP07, CNPP09, PPP09}.

We are interested in the behavior of the graph $S_n(r_n,c_n)$ for large values of $n$. When we say that a property of the graph holds \emph{with high probability} (whp), we mean that the probability that the property does not hold is bounded by a function of $n$ that goes to zero as $n\to \infty$. Equivalently we say that a sequence of random events $E_n$ occurs with high probability if $\lim_{n\to\infty} \Pr{E_n} = 1$. There are two independent sources of randomness in the definition of the random graph $S_n(r_n,c_n)$. One comes from the random underlying geometric graph $G_n(r_n)$ and the other from the choice of the $c_n$ neighbors of each vertex. 

Since we are interested in connectivity of $S_n(r_n,c_n)$, a minimal
requirement is that $G_n(r_n)$ should be connected. It is well known
that the connectivity threshold of $G_n(r_n)$ is $\gamma^*
\sqrt[d]{\log n/n}$ where $\gamma^* = 1/\sqrt[d]{\Vol B(0,1)}$, where
$\Vol$ is the Lebesgue measure and $B(0,1)$ is the unit ball in
$\R^d$. See \cite{Pen97,GK98} or Theorem 13.2 in \cite{Pen03}. This
means that $G_n$ is connected with high probability if $r_n$ is at least
$\gamma \sqrt[d]{\log n/n}$ where $\gamma > \gamma^*$ while $G_n$ is
disconnected with high probability if $r_n$ is less than $\gamma
\sqrt[d]{\log n/n}$ where now $\gamma < \gamma^*$. We always consider
values of $r_n$ above this level.

When $r_n=r$ is constant, the geometry has very little influence: For instance, \citet*{DJHPS05} showed that when $r_n=r$ is independent of $n$, $S_n(r,2)$ is connected with high probability. The case when $r_n$ is small is a more delicate issue, since the geometry now plays a crucial role. \citet*{CNPP09} proved that in dimension $d=2$ there exist constants $\gamma_1,\gamma_2$ such that if $r_n \geq \gamma_1 \sqrt{\log n/n}$ and $c_n \geq \gamma_2 \log(1/r_n)$, then $S_n(r_n,c_n)$ is connected with high probability. 

Arguably the most interesting values for $r_n$ are those just above the connectivity threshold for the underlying graph $G_n(r_n)$, that is, when $r_n$ is proportional to $\sqrt[d]{\log n/n}$. The results of \citet{CNPP09} show that for such values of $r_n$, connectivity of $S_n(r_n,c_n)$ is guaranteed, with high probability, when $c_n$ is a sufficiently large constant multiple of $\log n$. In this paper we show that this bound can be improved substantially. For the given choice of $r_n$, there is a critical $c_n$ for connectivity. It is quite easy to show that no connectivity can take place (whp) for constant $c_n$, and that for $c_n \geq \lambda \log n$ for a sufficiently large $\lambda$, the graph is connected whp (because the maximal cardinality of any ball of radius $r$ is whp $O(\log n)$). The objective of this paper is to nail down the precise threshold. Our main result is the following theorem.

\begin{theorem}\label{thm:main}
There exists a finite constant $\gamma^{**}$ such that for all $\gamma \geq \gamma^{**}$, and with 
\[
r_n = \gamma \left( \frac{\log n}{n} \right)^{1/d},
\]
we have, for all $\eps \in (0,2)$,
\[
\lim_{n \to \infty} \Pr{S_n(r_n,c_n) ~\text{\rm is connected} }
= \begin{cases}
    1 & \text{ if } c_n \ge \sqrt{ \frac{(2+\eps) \log n}{\log \log n} }, \\
    0 & \text{ if } c_n \le \sqrt{ \frac{(2-\eps) \log n}{\log \log n} }. 
\end{cases}
\]
\end{theorem}
\noindent 
It might be a bit surprising that the threshold is virtually independent of $\gamma$. The threshold (in $c$) is also independent of the dimension $d$. This is probably less surprising since $c$ counts a number of neighbors and the number of visible points in a ball is of order $\log n$, independently of $d$, for the range of $r$ we consider.

The structure of the paper is the following: In Section~\ref{sec:lower} we prove a lower bound on the critical value of $c_n$
needed to obtain a connected graph whp given a value of $r_n$ in the
range where connectivity could be achieved. In Section~\ref{sec:upper}
we show that $S_n(r_n,c_n)$ is connected whp where $r_n$ is
proportional to $\sqrt[d]{\log n/n}$ and $c_n$ is just above the
corresponding value obtained in Section~\ref{sec:lower} nailing down
the precise threshold in that case. Finally in Section~\ref{sec:diameter} we obtain an upper bound on the diameter of
$S_n(r_n,c_n)$ for the same values of $r_n$ as in Section~\ref{sec:upper} but with a larger value of $c_n$. In particular, 
we show that if $c_n$ is a sufficiently large constant times
$\sqrt{\log n}$ then the diameter of $S_n(r_n,c_n)$ is $O(1/r_n)$
which is the same order of magnitude as for the underlying random
geometric graph.

A final notational remark: To ease the reading for the rest of the paper we omit the subscript $n$ in the parameters $r$ and $c$ as well as in most of the events and sets we define that depend on $n$.

\section{A lower bound for connectivity on the whole range}\label{sec:lower}

The aim of this section is to prove a lower bound on the value of $c$ needed to obtain connectivity whp for a given value of $r$. First we will need the following lemma on the regularity of uniformly distributed points. Let $N(A)= \sum_{i=1}^n \ind{[X_i\in A]}$ be the number of data points in a set $A\subset [0,1]^d$. We  consider $\gamma^{**} > \gamma^{*}$ to provide a sufficient margin of play.

\begin{lemma}[Density Lemma]\label{lem:density}
Let $\gamma^{**} = 2\sqrt[d]{3}$ and $f(x)=x\log x - x + 1$. Define $\alpha,\beta$ to be the solutions to $f(x)=1/2$ smaller and greater than $1$ respectively. Grid the cube $[0,1]^d$ using cells of side length $\ell = (1/3)\flr{1/r}^{-1}$. If
\[
\gamma^{**}\left(\frac{\log n}{n}\right)^{1/d} < r < 1~,
\]
then the following event occurs whp:
\[
N(C) \in [\alpha n\ell^d, \beta n\ell^d] \quad\text{for every cell $C$}.
\]
\end{lemma}

\begin{proof}
We use the binomial Chernoff bound: If $\xi\sim \Bin(n,p)$ and $t>0$ then
\[
\max(\Pr{\xi \leq t}, \Pr{\xi \geq t}) \leq \exp{t-np-t\log\left(\frac{t}{np}\right)},
\]
for reference see \cite{Che52, JLR00}.
Given a fixed cell $C$, the number of data points $N(C)$ is distributed as $\Bin(n,\ell^d)$. Thus, writing $f(x)=x-1-x\log x$, we have
\begin{align*}
\Pr{N(C) \leq \alpha n\ell^d} &\leq \exp{-f(\alpha)n\ell^d}, \\
\Pr{N(C) \geq \beta n\ell^d} &\leq \exp{-f(\beta)n\ell^d}.
\end{align*}
Define the event $D(C) = \{N(C) \in [\alpha n\ell^d, \beta n\ell^d]\}$. We can apply a union bound over all the cells to obtain
\[
\Pr{\bigcup_C D(C)^c} \leq \sum_C \Pr{D(C)^c} \leq \sum_C 2e^{-n\ell^d/2} \leq \ell^{-d} 2e^{-n\ell^d/2} \to 0,
\]
because we have that $\ell^{-d}=O(r^{-d})=O(n/\log n)$ and $n\ell^d \geq n(r/3)^d \geq 2\log n$ so that $e^{-n\ell^d/2}=O(1/n)$.
\end{proof}

\begin{remark}
When the event from the previous lemma holds then we have $N(B(x,r)) \in [2^d\alpha n\ell^d , 7^d \beta n\ell^d]$ for any point $x\in [0,1]^d$ where we write $B(x,r)=\{y: \norm{x-y}< r\}$. This is because $B(x,r)$ always contains at least $2^d$ cells and is covered by at most $7^d$ cells. See Figure \ref{fig:covering}.
\end{remark}

\begin{figure}[h]
\centering
\includegraphics[height=3cm]{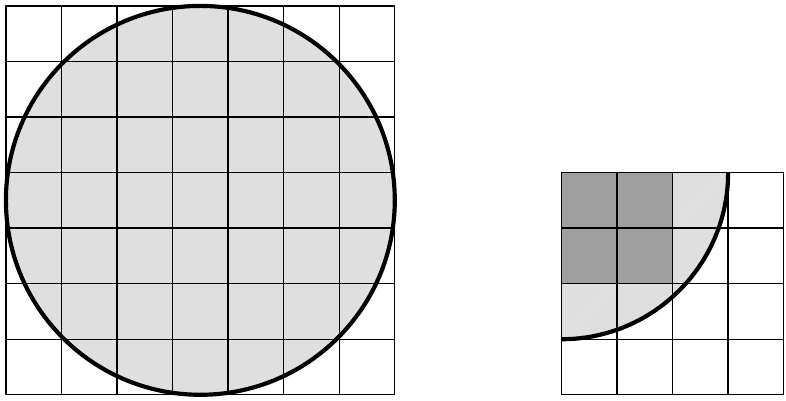}
\caption{\label{fig:covering}
A ball $B(x,r)$ in dimension $d=2$ covered by $7^d$ and covering at least $2^d$ cells. }
\end{figure}

The next theorem shows that for any value of $r$ above the
connectivity threshold of the random geometric graph one cannot hope
that $S_n$ is connected unless $c$ is at least of the order of
$\sqrt{\log n/\log nr^d}$. In particular, when $r$ is just above the
threshold (i.e., it is proportional to $\sqrt[d]{\log n/n}$) then $c$
must be at least of the order of $\sqrt{\log n/\log \log n}$. We say
that the data points at distance less than $r$ from $X_i$ are the
\emph{visible neighbors} of $i$ (i.e., the neighbors of $i$ in $G_n$)
and that $B(X_i,r)$ is the \emph{visibility ball} of $i$.  Note that
the following result implies the lower bound of Theorem \ref{thm:main}.

\begin{theorem}\label{thm:lower}
Let $\eps \in (0,1)$ and $\lambda \in [1,\infty]$ be such that
\[
\gamma^{**} \left(\frac{\log n}{n}\right)^{1/d} < r < 1,
\qquad
\frac{\log nr^d}{\log\log n} \to \lambda
\qquad\text{and}\qquad
c=\flr{\sqrt{(1-\eps)\left(\frac{\lambda}{\lambda-1/2}\right)\frac{\log n}{\log nr^d}}}.
\] 
Then $S_n(r,c)$ is not connected whp. (In the case of $\lambda=\infty$, 
we define $\lambda/(\lambda-1/2)=1$.)
\end{theorem}

Note that in the range of $r$ considered, we do have $\lambda\ge 1$.

\begin{proof}
We show that there exists an isolated $(c+1)$-clique whp. To do this we tile the unit cube $[0,1]^d$ with cells of side $\ell = (1/3) \flr{1/r}^{-1}$. We repeatedly use the fact that $r/3 \leq \ell < r$. We group the cells in large cubes of $7^d$ cells (which we denote by $L_{\jj}$) and we denote by $C_{\jj}$ the center cell of each large cube. See Figure \ref{fig:largecell}. Let $m = \flr{1/(7\ell)}$ be the number of large cubes we can fit on a side of $[0,1]^d$. Write $\jj=(j_1,\dots,j_d)$ where $0\leq j_k \leq m$ and define
\[
C_{\jj} = \prod_{k=1}^d \big[(7j_k+3)\ell, (7j_k+4)\ell \big]
\qquad\text{and}\qquad
L_{\jj} = \prod_{k=1}^d \big[7j_k\ell, (7j_k+7)\ell \big].
\]
\begin{figure}[h]
\centering
\includegraphics[height=6cm]{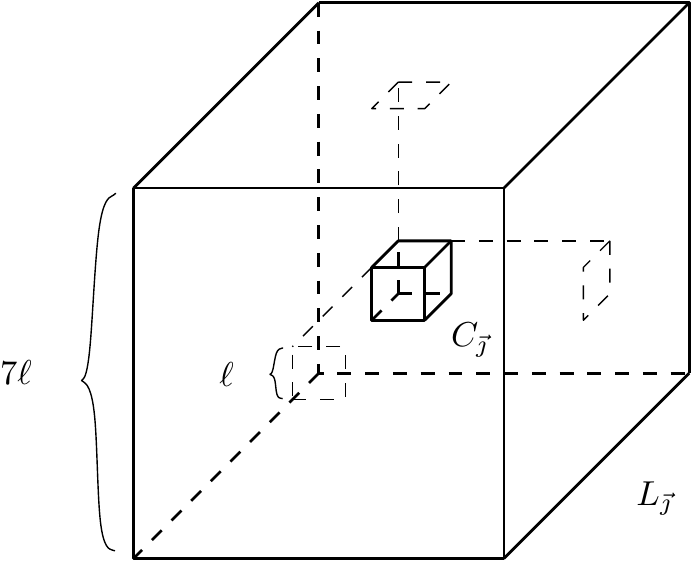}
\caption{\label{fig:largecell}
A large cube $L_{\jj}$ and its center cell $C_{\jj}$.}
\end{figure}

We slightly abuse notation and write $i\in C_{\jj}$ (or $L_{\jj}$) to mean $X_i \in
C_{\jj}$ (or $L_{\jj}$). Define the random family of sets $\F_{\jj} =
\{Q\subset V: |Q|=c+1,\; Q \subset C_{\jj}\}$. Denote by $E(Q)$ the
event that $Q$ forms an isolated clique in $S_n$. Consider the event
$E_{\jj}$ that there exists an isolated $(c+1)$-clique in $C_{\jj}$,
i.e., $E_{\jj} = \cup_{Q\in \F_{\jj}} E(Q)$. Let $D$ be the event that
all cells of side $\ell$ have cardinality in the range described in
Lemma \ref{lem:density}. By conditioning on all the data points
$X_1,\dots,X_n$ the only randomness we consider are the choices of
each node among their visible neighbors. Thus, the probability that
there is no isolated $(c+1)$-clique may be bounded as
\begin{align*}
\Pr{\cap_{\jj} E_{\jj}^c}
&\leq \Pr{D^c} + \Pr{\cap_{\jj} E_{\jj}^c, D} \\
&= \Pr{D^c} + \Ex{\Pr{\cap_{\jj} E_{\jj}^c, D \mid X_1,\dots,X_n} \Big.} \\
&= \Pr{D^c} + \Ex{\Pr{\cap_{\jj} E_{\jj}^c \mid X_1,\dots,X_n} \cdot \ind{D} \Big.} \\
&= \Pr{D^c} + \Ex{\prod_{\jj} \Pr{E_{\jj}^c \mid X_1,\dots,X_n} \cdot \ind{D} } 
\end{align*}
where the last equality comes from the fact that the events $E_{\jj}$ are independent conditionally on $X_1,\dots,X_n$ since $E_{\jj}$ only involves the choices of the indices with data points in $L_{\jj}$ (if $\norm{x-y}< r$ for some $y\in C_{\jj}$ then $x\in L_{\jj}$) and these cubes are disjoint (except, possibly, on their boundary which has measure zero, so with probability $1$ no data point will lie there). If we find a uniform bound $p_n\in [0,1]$ such that for every $\jj$
\[
\Pr{E_{\jj} \mid X_1,\dots, X_n} \geq p_n\ind{D},
\]
since $D$ holds whp then we have
\begin{align*}
\Pr{\cap_{\jj} E_{\jj}^c}
&\leq o(1) + (1-p_n)^{m^d} \\
&\leq o(1) + e^{-m^dp_n}.
\end{align*}
Notice that the upper bound goes to zero if $m^dp_n\to \infty$. In that case, $\cup_{\jj} E_{\jj}$ happens whp and the proof is complete. In Lemma \ref{lem:pnbound} we show can define $p_n = p_n^{(1)} - p_n^{(2)}$ where
\begin{align*}
p_n^{(1)} &= \exp{c^2(\log c -\log nr^d) + \log nr^d - o(\log n) \Big.}, \\
p_n^{(2)} &= \exp{2c^2(\log c - \log nr^d) + 2\log nr^d + o(\log n) \Big.}.
\end{align*}
Recall that $m=\flr{1/7\ell} = \Theta(r)$. To finish the proof it suffices to show that 
\[
m^d p_n^{(1)} \to \infty \qquad\text{and}\qquad m^d p_n^{(2)} \to 0.
\]
For the first limit we have, when $\lambda<\infty$,
\begin{align*}
m^d p_n^{(1)}
&\geq \exp{c^2(\log c - \log nr^d) + \log n - o(\log n) \Big.} \\
&\geq \exp{\frac{(1-\eps)\lambda\log n}{(\lambda-1/2)\log nr^d}\left(\frac{1}{2}\log\log n - \log nr^d\right) + \log n - o(\log n) } \\
&\geq \exp{\frac{(1-\eps)\lambda}{\lambda-1/2}\left(\frac{1}{2\lambda} - 1\right)\log n + \log n - o(\log n) } \\
&\geq \exp{-(1-\eps)\log n + \log n - o(\log n) \Big.} \\
&\geq \exp{\eps \log n + o(\log n) \Big.} \to \infty.
\end{align*}
When $\lambda=\infty$, the proof is analogous, if we substitute $\frac{\lambda}{\lambda-1/2}$ by $1$ and $\frac{1}{2\lambda}$ by $0$ in the previous equation. For the second limit we have
\begin{align*}
m^d p_n^{(2)}
&\leq \exp{2c^2 (\log c - \log nr^d) + 3/2 \log n + o(\log n) \Big.} \\
&\leq \exp{-2(1-\eps)\log n + 3/2 \log n + o(\log n) \Big.} \\
&\leq \exp{- (1/2-\eps)\log n + o(\log n) \Big.} \to 0,
\end{align*}
where we used the fact 
\[
\frac{\log nr^d}{\log n} = \frac{(1-\eps)\lambda}{(\lambda-1/2)c^2} \leq \frac{1}{2}. \qedhere
\]
\end{proof}
\medskip

We now show how to find $p_n$ with the desired properties.

\begin{lemma}\label{lem:pnbound}
With the notation from Theorem \ref{thm:lower} above we have that 
$\Pr{E_{\jj} \mid X_1,\dots, X_n} \geq p_n\ind{D}$, where $p_n = p_n^{(1)} - p_n^{(2)}$ and
\begin{align*}
p_n^{(1)} &= \exp{c^2(\log c -\log nr^d) + \log nr^d - o(\log n) \Big.}, \\
p_n^{(2)} &= \exp{2c^2(\log c - \log nr^d) + 2\log nr^d + o(\log n) \Big.}.
\end{align*}
\end{lemma}
\begin{proof}
We have that $E_{\jj} = \cup_{Q\in \F_{\jj}} E(Q)$ which allows us to use the following lower bound from inclusion--exclusion
\begin{equation}\label{eq:sumEj}
\Pr{E_{\jj}\mid X_1,\dots,X_n} \geq \sum_{Q\in \F_{\jj}} \Pr{E(Q)\mid X_1,\dots,X_n} 
- \mathop{\sum_{Q,Q'\in \F_{\jj}}}_{Q \neq Q'} \Pr{E(Q)\cap E(Q')\mid X_1,\dots,X_n}.
\end{equation}
Using the fact that $c^2 = O(\log n / \log nr^d) = O(\log n / \log\log n) = o(\log n)$ we have the following bounds for the size of $\F_{\jj}$ when $D$ holds:
\begin{align}
|\F_{\jj}\,| &\geq \binom{\ceil{\alpha n\ell^d}}{c+1} 
\geq \left(\frac{\alpha nr^d/3^d}{c+1}\right)^{c+1}
\geq \exp{-c(\log c - \log nr^d) + \log nr^d - o(\log n) \Big.}, \label{eq:lowerFj} \\
|\F_{\jj}\,| &\leq \binom{\flr{\beta n\ell^d}}{c+1} 
\leq \left(\frac{e\beta nr^d}{c+1}\right)^{c+1}
\leq \exp{-c(\log c - \log nr^d) + \log nr^d + o(\log n) \Big.}. \label{eq:upperFj}
\end{align}

We first bound the first sum in \eqref{eq:sumEj}. Define the event $E_1(Q)$ of $Q$ being a clique (i.e., every point $j\in Q$ chooses the remaining $c$ points $i$ with $i \in Q$, $i\neq j$ to link to), and the event $E_2(Q)$ of $Q$ being avoided in $L_{\jj}$ (i.e., every $i \in L_{\jj}\setminus Q$ avoids choosing the nodes in $Q$ as an endpoint of any of its $c$ links). Clearly if $Q\in \F_{\jj}$ then $E(Q) = E_1(Q) \cap E_2(Q)$ because the only nodes that can link to $Q$ are in $L_{\jj}$ in that case. Furthermore, conditionally on $X_1,\dots,X_n$, the events $E_1(Q)$ and $E_2(Q)$ are independent (because they involve the choices of disjoint sets of indices). So, we have
\begin{align*}
\Pr{E(Q) \mid X_1,\dots,X_n} 
&= \Pr{E_1(Q)\cap E_2(Q)\mid X_1,\dots,X_n} \\
&= \Pr{E_1(Q)\mid X_1,\dots, X_n} \cdot \Pr{E_2(Q)\mid X_1,\dots, X_n} \\
&\ge \prod_{i \in Q} \prod_{k=0}^{c-1} \left(\frac{c-k}{N(B(X_i,r))-k}\right) \cdot \prod_{i \in L_{\jj}\setminus Q}\prod_{k=0}^{c-1} \left(1-\frac{c+1}{N(B(X_i,r))-k}\right).
\end{align*}
When $D$ holds we have
\begin{align*}
\prod_{i \in Q} \prod_{k=0}^{c-1} \left(\frac{c-k}{N(B(X_i,r))-k}\right)
&\geq \left(\prod_{k=0}^{c-1} \frac{c-k}{7^d\beta n\ell^d-k}\right)^{c+1} \\
&\geq \left(\frac{c!}{(7^d\beta nr^d)^c}\right)^{c+1} \\
&\geq \left(\frac{c}{e7^d\beta nr^d}\right)^{c(c+1)} \\
&\geq \exp{(c^2+c)(\log c - \log nr^d)-o(\log n)\Big.}.
\end{align*}
Since we can cover $L_{\jj}$ with $7^d$ cells of side length $\ell$ we also have, for $n$ large enough,
\begin{align*}
\prod_{X_i \in L_{\jj}\setminus Q}\prod_{k=0}^{c-1} \left(1-\frac{c+1}{N(B(X_i,r))-k}\right)
&\geq \prod_{k=0}^{c-1} \left(1-\frac{c+1}{2^d \alpha n\ell^d-k}\right)^{N(L_{\jj})} \\
&\geq \left(1-\frac{2c}{2^d \alpha n\ell^d}\right)^{c7^d \beta n\ell^d} \\
&= \exp{\frac{2c^2 7^d \beta}{2^d \alpha}\cdot\frac{2^d\alpha n\ell^d}{2c} \log\left(1-\frac{2c}{2^d\alpha n\ell^d}\right)} \\
&\geq \exp{\frac{-4c^2 7^d \beta}{2^d \alpha}} \\
&\geq \exp{-o(\log n)\Big.},
\end{align*}
where we used that $c=o(\sqrt{\log n})$ and $n\ell^d = \Omega(\log n)$ so $k < c < 2^d\alpha n\ell^d / 4$ for $n$ large. Therefore, we have for $Q \in \F_{\jj}$
\[
\Pr{E(Q) \mid X_1,\dots,X_n} \geq \exp{(c^2+c)(\log c - \log nr^d)-o(\log n)\Big.} \cdot \ind{D}. 
\]
Using the bound we obtained and the lower bound from \eqref{eq:lowerFj} we get
\begin{align*}
\sum_{Q\in \F_{\jj}} \Pr{E(Q) \mid X_1,\dots,X_n} 
&\geq |\F_{\jj}| \cdot  \exp{(c^2+c)(\log c - \log nr^d)-o(\log n)\Big.} \cdot \ind{D} \\
&\geq \exp{c^2(\log c -\log nr^d) + \log nr^d - o(\log n) \Big.} \cdot \ind{D}.
\end{align*}

Next, we bound the second sum in \eqref{eq:sumEj}. Note that the events $E(Q)$ and $E(Q')$ can only happen at the same time if $Q\cap Q'=\emptyset$. In that case for $Q,Q' \in \F_{\jj}$ we obtain
\begin{align*}
\Pr{E(Q)\cap E(Q') \mid X_1,\dots X_n} 
&\le \Pr{E_1(Q)\cap E_1(Q') \mid X_1,\dots X_n} \\
&\leq \prod_{i \in Q} \prod_{k=0}^{c-1} \frac{c-k}{N(B(X_i,r))-k} \cdot \prod_{i \in Q'} \prod_{k=0}^{c-1} \frac{c-k}{N(B(X_i,r))-k} \\
&\leq \exp{2(c^2+c)(\log c - \log nr^d) + o(\log n)} \cdot \ind{D},
\end{align*}
where the inequality comes from the fact that when $D$ holds and $Q,Q' \in \F_{\jj}$,
\begin{align*}
\prod_{i \in Q} \prod_{k=0}^{c-1} \frac{c-k}{N(B(X_i,r))-k}
&\leq \left(\prod_{k=0}^{c-1} \frac{c-k}{2^d \alpha n\ell^d-k}\right)^{c+1} \\
&\leq \left(\frac{c}{2^d \alpha nr^d/3^d}\right)^{c(c+1)} \\
&\leq \exp{(c^2+c)(\log c - \log nr^d) + o(\log n)}.
\end{align*}
Therefore, for the second sum in \eqref{eq:sumEj}, using the upper bound in \eqref{eq:upperFj} we get
\begin{align*}
\sum_{Q\neq Q'\in \F_{\jj}} \Pr{A(Q)\cap A(Q') \mid X_1,\dots X_n} 
&\leq |\F_{\jj}|^2 \cdot \exp{2(c^2+c)(\log c - \log nr^d) + o(\log n)} \cdot \ind{D} \\
&\leq \exp{2c^2(\log c - \log nr^d) + 2\log nr^d + o(\log n) \Big.} \cdot \ind{D}.
\end{align*}
Thus we can take $p_n^{(1)}$ and $p_n^{(2)}$ as stated.
\end{proof}

\section{Connectivity near the critical radius}\label{sec:upper}

In this section we prove the remaining part of Theorem \ref{thm:main}. We consider $r = \gamma \sqrt{\log n/n}$ with $\gamma > \gamma^{**}$. We only need to prove that $S_n$ is connected whp when $c$ is above the threshold since Theorem \ref{thm:lower} implies that $S_n$ is disconnected whp when $c$ is below it.

\begin{theorem}\label{thm:upper}
Let $\eps \in (0,2)$,  $\gamma \geq \gamma^{**}$ and suppose that 
\[
r=\gamma \left(\frac{\log n}{n}\right)^{1/d} 
\quad\text{and}\qquad
c \ge  \sqrt{\frac{(2+\eps)\log n}{\log\log n}},
\]
Then $S_n(r,c)$ is connected whp.
\end{theorem}

We first give a high-level proof using a combinatorial argument which reduces the problem of connectivity to the occurrence of four properties that will be shown to hold in a second part.

We tile the unit cube $[0,1]^d$ into cubes of side length $\ell = (1/3) \flr{1/r}^{-1}$. We then consider cells of side $3\ell$ each one consisting of $3^d$ cubes. A cell is interconnected and colored black if all the vertices in it are connected to each other without ever using an edge that leaves the cell. The other cells are initially colored white. Two cells are connected if they are adjacent (they share a $(d-1)$-dimensional face) and there is an edge of $S_n$ that links a vertex in one cell to a vertex in the other cell. Two cells are $*$-connected if they are adjacent (they share at least a corner) and there is an edge of $S_n$ binding one vertex of each cell. 

Consider the following events:
\begin{enumerate}[(i)]
\item All cells in the grid are occupied and connected to their $2d$ neighbors.
\item The largest $*$-connected component of white cells has cardinality at most $q$.
\item The smallest connected component of $S_n$ is of size at least $s$.
\item Each grid cell contains at most $\lambda \log n$ vertices.
\end{enumerate}

\begin{proposition}Suppose that (i)--(iv) above hold with high probability.
Assuming further that $q,s$ and $\lambda$ are positive functions of $n$ such that
\[
q < \frac 1{d+1}\frac{1}{r^{1-1/d}}
\qquad \text{and} \qquad
 \frac{s}{\lambda\log n} > ((d+1)q)^{d/(d-1)}~,
\]  
then the graph $S_n$ is connected.
\end{proposition} 
\begin{proof}
The proof uses a percolation style argument on the grid of cells. We define a \emph{black connector} as a connected component of black cells that links one side of the cube $[0,1]^d$ to the opposite side.

(a) There exists a black connector in the cell grid graph: Note that by a generalization of the celebrated argument of \citet{Kes80}, either there is a black connector, or there is a white $*$-connected component of cells that prevents this connection from happening (one of the two events must occur). In dimension $2$, this blocking $*$-connected component of white cells is a path that separates the two opposite faces of interest; in dimension $d$, the blockage must be a $(d-1)$-dimensional sheet (see also \citealp{Gri89,BR06}). In any case, the $*$-connected component of white cells, if it exists, must be of size at least $1/r$ in order to reach two opposite faces. Since the largest $*$-connected component of white cells has size at most $q < 1/r$, a black connector must exists. The black components of size less than $1/r$ are now recolored blue. Note that this leaves at least the black connector component, of size at least $1/r$.

(b) All remaining black cells are connected: Note that this implies that the corresponding vertices of $S_n$ belong to the same connected component. This collection of vertices of $S_n$ is called the \emph{black monster}. To see this, observe that if two black components of cells are not connected, then they must be separated by a $*$-connected component of white cells. Since every black component of cells has at least $1/r$ members, their boundary (neighboring cells or faces on the boundary of $[0,1]^d$) must be of size at least $1/r^{1-1/d}$ (see \cite{BL91} for the isoperimetric inequality). Among these $1/r^{1-1/d}$ members, at least a $1/(d+1)$ fraction must be part of a white $*$-connected component. (To see this, note that for each boundary face, there is
a unique white boundary cell in the direction of the opposing face of
$[0,1]^d$. This fashion, each white boundary cell is assigned to at most $d$ boundary cells and therefore the number of boundary faces is at most
$d$ times the number of white boundary cells.)
By assumption, $q(d+1) < 1/r^{1-1/d}$, and thus, no such separating white $*$-connected chain can exist.

\begin{figure}[h]
\centering
\includegraphics[height=8cm]{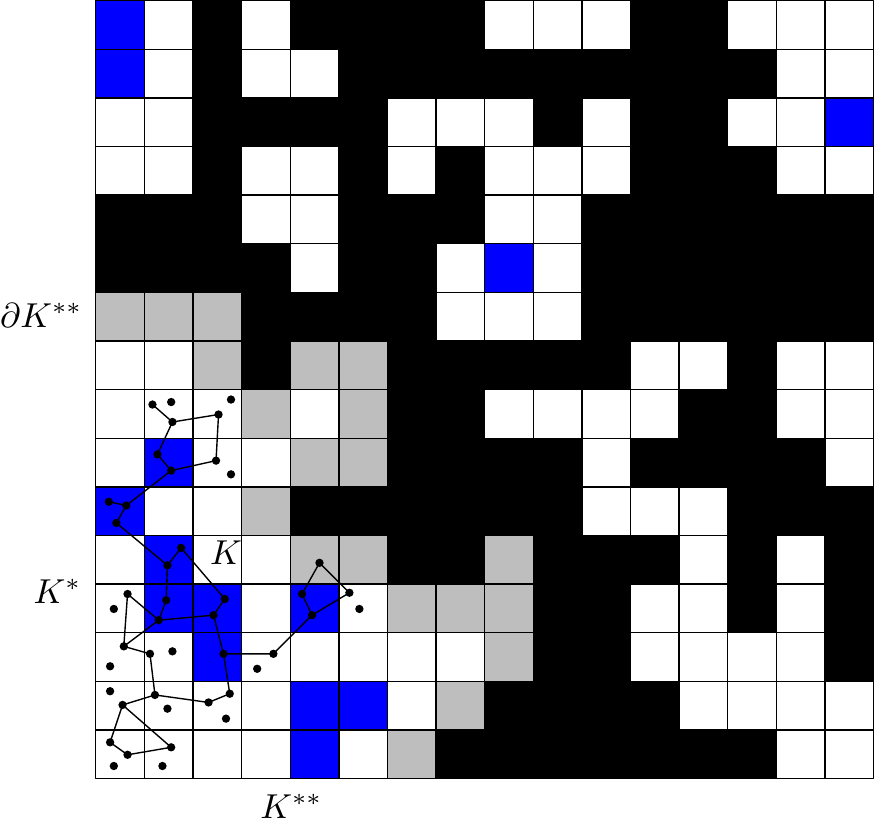}
\caption{\label{fig:percolation}
The component $K$ with its corresponding $*$-connected component of occupied cells $K^*$. We enlarge $K^*$ by adding all the connected white and blue cells to get $K^{**}$. The border cells of $K^{**}$ are in colored in gray.}
\end{figure}

(c) Each vertex connects to at least one vertex of the black monster: To prove this, consider any vertex $j$, outside of the black monster, and write $K$ for the component of $S_n$ it belongs to. If any vertex of $K$ lies in the black monster, then $j$ is connected to the black monster and we are done. So we now assume that all vertices of $K$ belong to white or blue grid cells. Adjacent vertices in $K$ lie in the same cell, or two $*$-adjacent cells. Let $K^*$ be the $*$-connected component of all grid cells visited by vertices of $K$. Enlarge $K^*$ by adding all grid cells that reach $K^*$ via a white $*$-connected chain of cells. The resulting $*$-connected component of white and blue cells is called $K^{**}$, see Figure \ref{fig:percolation}. By assumption, it contains at least $s/ (\lambda \log n)$ cells, since it covers the connected component $K$ of $S_n$ (by properties (iii) and (iv)). So we have exhibited a fairly large $*$-connected component of cells that are not black; the only issue is that it might not be fully white, and we wish to isolate a large \emph{white} $*$-connected component is order to invoke property (iii) for a contradiction. Call a cell of $K^{**}$ a border cell if one of its $2d$ neighbors in the grid is black. Clearly, border cells must be white, because no blue cell can have a black neighbor. Now, $K^{**}$ is surrounded either by border cells, or by pieces of the boundary of the cube. In any case, the boundary of $K^{**}$ (border cells, or boundary faces) has cardinality at least $|K^{**}|^{1-1/d}$, and (as we have already seen) at least a $1/(d+1)$ fraction of it must be made of a $*$-connected component of white cells.
By property (iii), this is impossible. This finishes the proof.
\end{proof}
\medskip

\noindent We now show (i) through (iv) in four lemmas, leaving the hardest one, (iii), for last. We show all these properties with $\lambda$ a sufficiently large constant depending upon $\gamma$, $q = 2 (\log n)^{2/3}$, and $s = \exp{ (\log n)^{1/3} }$, leaving wide margins. 
\medskip

\begin{lemma}[Part (iv)]
Each grid cell contains at most $\lambda \log n$ vertices with high probability, where $\lambda = 3^d\beta \gamma^d$.
\end{lemma}
\begin{proof}
By Lemma \ref{lem:density} we have that every cube of side $\ell$ has less than $\beta n\ell^d$ data points whp. Since $\ell < r$ this implies immediately that every cell contains at most $3^d\beta nr^d < 3^d\beta \gamma^d \log n$ points.
\end{proof}
\medskip

\begin{lemma}[Part (i)]\label{lem:connectedcells}
With high probability, all cells in the grid are occupied and connected to their $2d$ adjacent neighbors.
\end{lemma}
\begin{proof}
From Lemma \ref{lem:density}, we note that all cardinalities of the cubes of side $\ell$ are at least $\alpha n\ell^d$ (and at most $\beta n\ell^d$) whp. We assume that this event holds. We condition on any point set with this distributional property, leaving only the choices of the $c$ neighbors as a random event.
Consider two neighboring cells of side length $3\ell$ and the corresponding middle border cubes $C$ and $C'$ of side $\ell$ that are adjacent to each other, see Figure \ref{fig:adjacentcells}.

\begin{figure}[h]
\centering
\includegraphics[height=4cm]{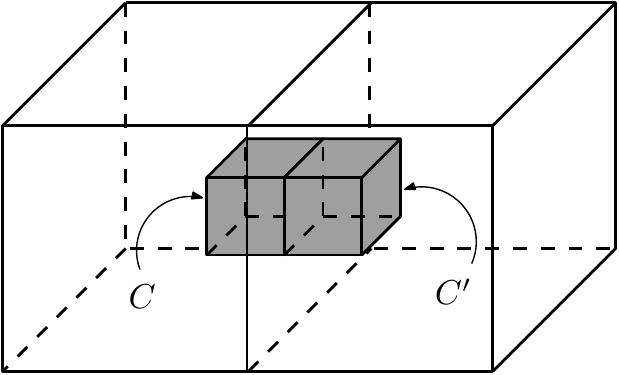} 
\caption{\label{fig:adjacentcells}
Two adjacent cells and the corresponding middle border cubes $C$ and $C'$.}
\end{figure}

The probability that the $c$ choices for a vertex in $C'$ miss any of the $\alpha n\ell^d$ points in $C$ (which are all in range $r$) is less than
\[
\left( 1 - \frac{\alpha}{7^d\beta} \right)^c \leq \exp{\frac{-\alpha c}{7^d\beta}}.
\]
since each ball $B(X_i,r)$ has cardinality at most $7^d n\ell^d$ (see the remark after Lemma \ref{lem:density}). By independence, the probability that \emph{all} points in $C'$ miss those in $C$ with their $c$ choices is not more than
\[
\prod_{i\in C'} \left( 1 - \frac{\alpha}{7^d\beta} \right)^c 
\leq  \left( 1-\frac{\alpha}{7^d \beta} \right)^{c \alpha n\ell^d}
\leq \exp{ \frac{-\alpha^2 c n\ell^d}{7^d\beta} }.
\]
Since there only a total of $O(r^{-d}) = o(n)$ cells, the union bound shows that the probability that there is an empty cell or that two neighboring cells do not connect tends to zero.
\end{proof}
\medskip

\begin{lemma}[Part (ii)]\label{lem:starsize}
The largest $*$-connected component of white cells has cardinality at most $q = 2 (\log n)^{2/3}$ whp.
\end{lemma}
\begin{proof}
We start by bounding the number of $*$-connected components of cells of a fixed size $k$. The number cells of $*$-adjacent to any fixed cell is at most $3^d$, so that by Peierls' argument \cite{Pei36} (see also Lemma 9.3 of \citealp{Pen03}), the number of $*$-connected components of size $k$ containing a specified cell is at most $2^{3^d k}$. Overall, the number of $*$-connected components of size $k$ is at most $n 2^{3^d k}$ (since there are at most $O(r^{-d}) = o(n)$ starting cells). 

By Lemma \ref{lem:density}, each cube of side $\ell$ has whp a cardinality contained in $[\alpha n\ell^d, \beta n\ell^d]$, so we assume this event holds and we denote it by $D$. Assume that we can then show that the probability that a cell is white is at most $p$. In that case, the probability that there is a $*$-connected component of size $k$ or larger is not more than
\begin{equation}\label{eq:peierls}
n 2^{3^d k} p^k,
\end{equation}
by the union bound and because the colors of the cells are independent, given the location of the data points. If we can show that 
\[
p \leq \exp{ - (\log n)^{1/3} }
\]
then $k = 2 \log^{2/3} n$ suffices to make the probability bound \eqref{eq:peierls} tend to zero.

We first prove that for $n$ large enough, the probability that a specified cell is white is at most $\exp{-(\log n)^{1/3}}$. By the preceding arguments, this will complete the proof of the lemma. Recall that a cell is colored white if the graph induced by the vertices lying inside the cell is not connected. We show that, with high probability, any two points in the cell are actually linked by a path of length at most six.

Consider now a fixed cell and take any vertex $v$ inside the cell. Let $V'$ be the subset of the $c$ neighbors of $v$ that fall within the cube $T$ of side $\ell$ located at the center of the cell. Consider then all $c$ choices of the vertices in $V'$ that fall in $T$ as well, and that are not in $\{v\} \cup V'$. Call that second collection $V''$. We show that with high probability, all the remaining vertices select at least one point from $\{v\}\cup V' \cup V''$. Each of the remaining vertices selects in any of its $c$ choices a vertex in $\{ v \} \cup V' \cup V''$ with probability at least 
\[
\frac{1 + |V'|+ |V''|}{7^d\beta n\ell^d}
\]
because the possible selection area for a vertex covers at most $7^d$ squares of side $\ell$. The probability that some vertex does not select any neighbor from $\{ v \} \cup V' \cup V''$ is at most
\begin{align*}
\sum_{w \notin \{ v \} \cup V' \cup V''}  \left( 1 - \frac{1 + |V'|+ |V''|}{7^d\beta n\ell^d} \right)^c
&\leq 3^d\beta n\ell^d \times \left( 1 - \frac{1 + |V'|+ |V''|}{7^d\beta n\ell^d} \right)^c\\
&\leq 3^d\beta n\ell^d \exp{ - \frac{|V''| c}{7^d\beta n\ell^d} }.
\end{align*}
If all vertices select a neighbor inside $\{ v \} \cup V' \cup V''$, then clearly, all vertices are connected (and within distance six of each other, pairwise), and the cell is black. As a consequence, the probability of a having a white cell given the event $D$ is thus bounded from above by
\[
\Pr{ |V''| \leq \delta c^2 \mid D } + 3^d\beta n\ell^d \exp{ - \frac{\delta c^3}{7^d\beta n\ell^d} },
\]
where $\delta > 0$ is a constant to be selected later. Note that, for any $\delta > 0$, the second term in the upper bound is smaller than $\exp{ - (\log n)^{1/3} }$ for all $n$ large enough.

Finally, then, we consider $V'$ and $V''$ and condition on the event $D$ of all squares having cardinalities in the range given above. This implies that $N(T) \geq \alpha n\ell^2$. Now, for $V''$ to be small, one of the following events must occur: either $V'$ is small, or $V'$ is not small but $V''$ is small. Note that by definition $|V'|$ is stochastically larger than a 
\[
\Bin \left( c , \frac{N(T)-c}{7^d\beta n\ell^d} \right).
\]
which for $n$ large enough is stochastically larger than a
\[
\Bin \left( c , \frac{\alpha}{2 \beta 7^d} \right).
\]
We write $Z$ for a random variable distributed as above. We repeat a similar argument and note that $|V''|$ is stochastically larger than a $Z$-fold sum of independent binomial random variables, each of parameters $c$ and $(N(T)-c-c^2)/7^d\beta n\ell^d$. Thus, assuming $D$ and for $n$ large enough, $|V''|$ is stochastically larger than a
\[
\Bin \left( Zc , \frac{\alpha}{2\beta 7^d} \right).
\]

So gathering the preceding observations and taking $\delta= (\alpha / 4\beta 7^d)^2$, we obtain
\begin{align*}
\Pr{ |V''| \leq \delta c^2 \mid D}
&\leq \Pr{ Z \leq \frac{\alpha c}{4\beta 7^d} } + \Pr{ Z \geq \frac{\alpha c}{4\beta 7^d} ,\; \Bin \left( Zc , \frac{\alpha}{2\beta 7^d} \right) \leq \delta c^2 } \\
&\leq \Pr{ Z \leq \frac{\alpha c}{4\beta 7^d} } + \Pr{ \Bin \left( \flr{\frac{\alpha c^2}{4\beta 7^d}} , \frac{\alpha}{2\beta 7^d} \right) \leq \delta c^2 } \\
&\leq \exp{-\frac {\alpha c}{16\beta 7^d}}+\exp{-\frac{1}{2} \left(\frac{\alpha c}{4 \beta 7^d}\right)^2}.
\end{align*}
This shows that for $n$ large enough, $p \leq \exp{ -(\log n)^{1/3} }$, as required.
\end{proof}
\medskip

\begin{lemma}\label{lem:compsize}
The smallest connected component of $S_n$ is of size at least $s = \exp{(\log n)^{1/3}}$ whp.
\end{lemma}

\begin{proof}
It is in this critical lemma that we will use the full power of the threshold. The proof is in two steps. For that reason, we grow $S_n$ in stages. Having fixed $\eps$ in the definition of
\[
c = \sqrt{ \frac{(2+\eps) \log n}{\log\log n} },
\]
we find an integer constant $L$ (depending upon $\eps$ -- see further on), and let all vertices select their $c$ neighbors in rounds. In round one, each vertex selects
\[
\hat{c} = \sqrt{ \frac{(2+\eps/2) \log n}{\log\log n} }
\]
neighbors uniformly at random without replacement. Then, in each of the remaining $(c-\hat{c})/L$ rounds, each vertex chooses $L$ further neighbors within its range $r$, but this time independently and with replacement, with a possibility of duplication and selection of previously selected neighbors. This makes the graph less connected (by a trivial coupling argument), and permits us to shorten the proof. Note that
\[
\frac{c-\hat{c}}{L} = \Delta \sqrt{ \frac{2 \log n}{\log\log n} },
\qquad\text{where}\qquad
\Delta = \frac{ \sqrt{1+\eps/2} - \sqrt{1+\eps/4} }{L}.
\]

After the first (main) round, we will show that the smallest component is whp at least $\delta \log n$ in size, for a specific $\delta > 0$. We then show that whp, in each of the remaining rounds, each component joins another component, and thus the minimal component size doubles in each round. After the last round, the minimal component is therefore of size at least
\[
\delta \log n \times 2^{\frac{c-\hat{c}}{L}},
\]
which in turn is larger than $\exp{ (\log n)^{1/3} }$ for all $n$ large enough.

So, on to round one. Let $N_h$ count the number of connected components of $S_n$ of size exactly $h$ obtained after round one. By definition, $N_h = 0$ for $h \leq \hat{c}$. We will exhibit $\delta, \delta' > 0$ such that, for all $n$ large enough,
\[
\sup_{\hat{c} < h \leq \delta \log n} \Ex{ N_h }  \leq n^{-\delta'}.
\]
By the first moment method, this shows that whp, the smallest component after round one is of size at least $\delta \log n$. 

We shall require a general notion of connectivity for sets of cells in the grid of size $\ell$. Given a finite symmetric set $A\subset \Z^d$ (i.e. $-x \in A$ for all $x \in A$) we say that a subset of cells is $A$-\emph{connected} if the corresponding integer coordinates in the grid $\{0,\dots,1/\ell \}^d$ induce a connected graph when we put an edge between $x$ and $y$ if and only if $y - x \in A$. 

Let $D$ be the event that each cube of side $\ell$ has cardinality in $[\alpha n\ell^d , \beta n\ell^d]$, where $\alpha$ and $\beta$ are as in Lemma \ref{lem:density}. So, $D$ happens whp. If $D$ holds, the number of sets of $h$ data points that can be connected is bounded from above as follows: A connected component of size $h$ occupies at most $h$ cells from the grid (one for each node) and these cells form an $A$-connected set where $A=\{-3,\dots,3\}^d$. Again, by Peierls argument the number of $A$-connected sets in the grid of size $j$ is at most $\ell^{-d} 2^{|A|j}$. See Lemma 9.3 of \cite{Pen03}. Therefore, we have at most
\[
\sum_{j=1}^h \ell^{-d} 2^{7^d j}\binom{\flr{j\beta n\ell^d}}{h}  
\leq 2\ell^{-d} 2^{7^d h} \left(\frac{eh\beta n\ell^d}{h}\right)^h
\leq n \left(2^{7^d}e\beta n\ell^d \right)^h
\]
sets of $h$ data points that can be connected.

Having a fixed set $\{i_1,\dots,i_h\}$ of indices, it can only form a connected component if all the $h$ points choose their neighbours among the remaining points in the set. Assuming $h < 2^d \alpha n\ell^d$, the probability of this is at most
\[
\prod_{j=1}^h\left(\prod_{k=1}^{\hat{c}} \frac{h-k}{N(B(X_{i_j},r))- k}\right)
\leq \left(\prod_{k=1}^{\hat{c}} \frac{h-k}{2^d \alpha n\ell^d - k}\right)^h 
\leq \left(\frac{h}{2^d \alpha n\ell^d}\right)^{\hat{c} h}.
\]
Therefore,
\[
\Ex{N_h \ind{D}} \leq n \left(2^{7^d}e\beta n\ell^d \right)^h \left(\frac{h}{2^d \alpha n\ell^d}\right)^{\hat{c} h}.
\]
We can rewrite the upper bound as
\[
f(h) = \exp{\log n + h \log\left( 2^{7^d}e\beta n\ell^d \right) + \hat{c}h \log \left( \frac{h}{2^d \alpha n\ell^d} \right) }.
\]
Note that $f(h)$ is decreasing for $h \leq \rho n\ell^d = \rho(\gamma/3)^d \log n$ where $\rho < 2^d \alpha/e$ because
\[
\frac{d}{dh}\Big( \log f(h)\Big) = \log\left( 2^{7^d}e\beta n\ell^d \right) + \hat{c} \log\left( \frac{eh}{2^d \alpha n\ell^d} \right) < 0
\]
for $n$ sufficiently large since $\hat{c} = \omega(\log n\ell^d)$. For such $\rho$, and $n$ large enough, the upper bound is thus maximal at $k = \hat{c}+1$. We have shown that
\begin{align*}
\Ex{N_h \ind{D}} &\leq f(\hat{c}+1) \\
& = \exp{\log n + (\hat{c}+1) \log\left( 2^{7^d}e\beta n\ell^d \right) + \hat{c}(\hat{c}+1) \log \left( \frac{\hat{c}+1}{2^d \alpha n\ell^d} \right) } \\
&= \exp{\hat{c}^2 \left( \log(\hat{c}+1) - \log(2^d \alpha n\ell^d) \right) + \log n + o(\log n) \Big.} \\
&= \exp{\frac{(2+\eps/2)\log n}{\log\log n} \left( \frac{1}{2}\log\log n - \log\log n \right) + \log n + o(\log n) \Big.} \\
&\leq \exp{-(1+\eps/4)\log n + \log n + o(\log n) \Big.} \\
&\leq \exp{-(\eps/4+o(1))\log n \Big.},
\end{align*}
so we have $\Ex{N_h \ind{D}} \leq n^{-\eps/5}$ for $n$ large enough. This means we can take $\delta = (2\gamma/3)^d \alpha/2e$ and $\delta' = \eps/5$. Define the event $E_h = [N_h > 0]$ of having a component of size $h$. Finally, the probability that a component of size at most $\delta \log n$ exists after round one is bounded from above by
\begin{align*}
\Pr{\bigcup_{h=\hat{c}+1}^{\delta \log n} E_h}
&\leq \Pr{D^c} + \sum_{h=\hat{c}+1}^{\delta \log n} \Pr{E_h \cap D} \\
&\leq \Pr{D^c} + \sum_{h=\hat{c}+1}^{\delta \log n} \Ex{N_h \ind{D}} \\
&\leq o(1) + (\delta \log n) n^{-\eps/5}  \to 0.
\end{align*}

For the final act, we tile the unit cube into cells of dimension
$\ell\times \cdots \times \ell$. Consider a connected component having size $t$ after round
one, where $\delta \log n \leq t \leq n^{1/4}$. Let the
vertices of this component populate the cells. The $i$-th cell
receives $n_i$ vertices from this component, and receives $m_i$
vertices from all other components taken together. The cell is colored
red if $n_i > m_i$ and blue otherwise. First note that not all cells
can be red, since that would mean that $t = \sum_i n_i \geq n/2$. In
one round, each vertex chooses $L$ eligible vertices in its
neighborhood independently and with replacement. Consider two
neighboring cells $i$ and $j$ (in any direction or diagonally) of
opposite color ($i$ is red and $j$ is blue). Conditional on $D$, the
probability that these cells do not establish a link between the size
$t$ component and any of the other components is at most
\begin{align*}
\left( 1 - \frac{n_i}{7^d\beta n\ell^d} \right)^{L m_j}
&\leq \exp{ - \frac{ L n_i m_j }{ 7^d\beta n\ell^d } } \\
&\leq \exp{ - \frac{ L (\alpha n\ell^d / 2)^2 }{ 7^d\beta n\ell^d } } \\
&\leq \exp{ - \frac{ L \alpha^2 nr^d }{ 42^d\beta } } &\text{(recall $\ell > r/3$ and $r^d = \gamma^d \log n/n$),} \\
&=n^{- \frac{ L \alpha^2 \gamma^d }{ 42^d\beta } }. 
\end{align*}
Consider finally the situation that all cells are blue. Then the probability (still conditional on $D$) that no connection is established with the other components is not more than
\begin{align*}
\prod_i \left( 1 - \frac{n_i}{7^d\beta n\ell^d}\right)^{L m_i}
&\leq \exp{ - \sum_i \frac{ L n_i m_i }{ 7^d\beta n\ell^d } } \\
&\leq \exp{ - \frac{ L \alpha/2 }{ 7^d\beta } \sum_i n_i } \\
&\leq \exp{ - \frac{ L \alpha\;\delta \log n }{ 7^d 2 \beta } } \\
&=n^{ - \frac{ L \alpha \delta }{ 7^d 2 \beta } }. 
\end{align*}
Since there are not more than $n$ components to start with, the probability that any component of size between $\delta \log n$ and $n^{1/4}$ fails to connect with another one is bounded from above by
\[
n^{1-L\xi},
\]
where $\xi = \min\{ \alpha^2\gamma^d / 42^d \beta , \alpha\delta / 7^d 2\beta \}$. The probability that we fail in any of the $(c-\hat{c})/L$ rounds is at most equal to the probability that $D$ fails plus
\[
\frac{c-\hat{c}}{L} \times n^{1-L\xi} = o(1)
\]
by choosing $L$ large enough that $L\xi>1$. Thus, whp, after we are done with all rounds, the minimal component size in $S_n$ is at least
\[
\delta \log n \times 2^{\frac{c-\hat{c}}{L}}.
\]
This concludes the proof of Lemma \ref{lem:compsize}.
\end{proof}

\section{Upper bound for the diameter and spanning ratio}\label{sec:diameter}

In the previous sections, we have identified the threshold for connectivity near the critical radius. The connectivity is of course an important property, but the order of magnitude of distances in the sparsified $S_n$ graph should also be as small as possible. Here we show that in the same range of values of $r$ as in Theorem \ref{thm:main}, as soon as $c$ is of the order of $\sqrt{\log n}$ the diameter of the graph $S_n$ is $O(1/r)$ which is clearly best possible as even the diameter of $G_n(r)$ cannot be smaller than $\sqrt{2}/r$. This improves a result of \citet{PPP09}.

\begin{theorem}\label{thm:diam}
There exist a constant $\mu>0$ such that for any $\gamma > \gamma^{**}$, if 
\[
r=\gamma \left(\frac{\log n}{n}\right)^{1/d} 
\quad\text{and}\qquad
c \geq \mu \sqrt{\log n}~,
\]
then the diameter of $S_n$ is at most $36/r$, whp.
\end{theorem}

This also shows that the \emph{spanning ratio} is within a constant factor of the optimal. Given a connected graph embedded in the unit cube $[0,1]^d$ and two vertices $u$ and $v$ (points in space), let $d(u,v)$ denote the Euclidean distance between $u$ and $v$ when one is only allowed to travel in space along the straight lines between connected nodes in the graph (this is the intrinsic metric associated to the embedded graph). Of course $d(u,v)\ge \|u-v\|$, and one defines the spanning ratio as 
\begin{equation}\label{eq:def_spanning_rationI}
\sup_{u,v} \frac{d(u,v)}{\|u-v\|}~..
\end{equation}
One would ideally want the spanning ratio to be as close to one as possible. In the present case, this definition is not very relevant, since there is a chance that points that are very close in the plane are not connected by an edge. In particular one can show that, with probability bounded away from zero, there is a pair of points at distance $\Theta(n^{1/d})$ for which the smallest path along the edges is of length $\Theta(r)$, so that for some $\epsilon>0$, whp,
$$\liminf_{n\to\infty} \sup_{u,v\in S_n} \frac{d(u,v)}{\|u-v\|}\ge \epsilon (\log n)^{1/d}\to\infty.$$
(To see this, consider the event that for a point $X_i$, one other point falls within distance $\delta n^{-1/d}$ and there are no other points within distance $\epsilon (n/\log n)^{-1/d}$.)
This justifies introducing the constraint that the points in the supremum in \eqref{eq:def_spanning_rationI} be at least at distance $r$. Hence the following modified definition of spanning ratio:
$$\Gamma(S_n):=\sup_{i,j: \|X_i-X_j\|>r} \frac{d(X_i,X_j)}{\|X_i-X_j\|}.$$
Then Theorem~\ref{thm:diam} proves that there exists a constant $K$ independent of $n$ such that $\Gamma(S_n)\le K$.

The idea of the proof of Theorem~\ref{thm:diam} is the following. Partition the unit square into a grid of cells of side length $\ell = (1/3)\flr{1/r}^{-1}$. We show that, with high probability, any two points $i$ and $j$, such that $X_i$ and $X_j$ fall in the same cell, are connected by a path of length at most five. On the other hand, with high probability, any two neighboring cells contain two points, one in each cell, that are connected by an edge of $S_n$. These two facts imply the statement of the theorem. We prove them in two lemmas below. The bound for the diameter follows immediately from the fact that, with high probability, starting from any vertex, a point in a neighboring cell can be reached by a path of length $6$ and any cell can be reached by visiting at most $6 \flr{1/r}$ cells.

Just like in the arguments for the lower and upper bounds for connectivity, all we need about the underlying random geometric graph $G_n$ is that the points $X_1,\ldots,X_n$ are sufficiently regularly distributed. This is formulated as follows: A \emph{moon} is the intersection of two circles, one of radius $r$ and the other of radius $r/2$ such that their centers are within distance $5r/4$ (see Figure \ref{fig:moon}). Let $F$ be the event that for every moon $M$ with centers in $[0,1]^d$ and for every ball
\[
N(M) \in [\sigma\log n, \rho\log n]
\qquad\text{and}\qquad
N(B(x,r)) \in [\mu \log n, \nu \log n]
\]
where $\sigma,\rho,\mu,\nu$ are constants. By the remark after Lemma \ref{lem:density} we know $N(B(x,r)) \in [2^d\alpha n\ell^d, 7^d\beta n\ell^d]$ since $r/3 < \ell < r$ if $r = \gamma \sqrt[d]{\log n/n}$ we can take $\mu = (2\gamma/3)^d \alpha$ and $\nu = (7\gamma)^d\beta$. In the following lemma we show that there exist $\sigma<\rho$ such that the statement for moons holds whp. Thus, we have that $\lim_{n\to \infty} \Pr{F}=1$.

\begin{lemma}\label{lem:moons}
Let $X_1,\ldots,X_n$ be independent random points, uniformly distributed on $[0,1]^d$ and
\[
r=\gamma \left(\frac{\log n}{n}\right)^{1/d} 
\quad\text{with}\qquad  \gamma > \gamma^{**}.
\]
Then, there exists constants $0<\sigma<\rho<\infty$ such that, every moon $M$, with centers in $[0,1]^d$ has at least $\sigma \log n$ and at most $\rho \log n$ data points with high probability.
\end{lemma}
\begin{proof}
We want to show a concentration bound for the number of points lying in a moon. The bound easily obtained for any fixed moon using Chernoff's bound. The extension to any moon may be done using covering arguments. However, we find it more straightforward to rely on a generalization of the union bound to infinite classes of events. The bound is due to \citet{VC71} and uses the concept of VC-dimension.
	


Let $\Delta = \{ (x,y) \in [0,1]^d: \norm{x-y} < 5r/4\}$ and denote $M(x,y)=B(x,r)\cap B(y,r/2)$ to the moon with centers $x$ and $y$. Here the class of interest is $\mathcal{M} =\{M(x,y): (x,y)\in \Delta \}$. It is easy to see that the VC-dimension of $\mathcal{M}$ is at most $d+2$ since for any set of $d+2$ points of the unit cube, it is impossible to shatter $d+2$ points using sets in $\mathcal{M}$. Consider $n$ independent points uniformly distributed in $[0,1]^d$. Fix $\eps>0$, and for $(x,y) \in \Delta$, let $E_{x,y}$ be the event that the number $N_{x,y}=N(M(x,y))$ of points falling in the moon $M(x,y)$ satisfies $| N_{x,y} - \Ex{N_{x,y}} |\geq \eps \Ex{N_{x,y}}$. Then, a variant of the Vapnik-Chervonenkis inequality (Theorem~7 of \cite{BBL04}) implies
\begin{align*}
\Pr{\cup_{(x,y)\in \Delta} E_{x,y}}
&=\Pr{\exists x,y \in \Delta: | N_{x,y} - \Ex{N_{x,y}} | \geq \eps \Ex{N_{x,y}} \Big. } \\
&\leq 8 (2en)^{d+2}  \sup_{(x,y) \in \Delta} e^{-n \eps^2 \Ex{N_{x,y}} }\\
&\leq 8 (2en)^{d+2} \exp{-n^2 \eps^2 \theta_d 2^{-d} r^d},
\end{align*}
since any moon $M(x,y)$ has volume at least $\theta_d 2^{-d} r^d$, so that $\inf_{(x,y)\in \Delta} \Ex{N_{x,y}}\geq \theta_d 2^{-d} n r^d$. Taking, for instance, $\eps=1/2$, this yields that any moon $M(x,y)$, $(x,y)\in \Delta$ contains between $1/2 \cdot \theta_d 2^{-d} nr^d$ and $2\cdot 2^{-d} n r^{d}$ points with high probability, which completes the proof.
\end{proof}

The key lemma is the following.

\begin{lemma}\label{lem:withincell}
Fix $\{X_1,\ldots,X_n\}$ such that $F$ occurs with $\nu \geq 128$. Let $i,j$ be such that $X_i$ and $X_j$ fall in the same cell of the grid. If $c \geq  \sqrt{(\nu/2)\log n}$ then
\[
  \Pr{d(i,j) >5 \mid X_1,\dots, X_n} \leq \frac{1}{n}(1+o(1))~.
\]
where $d(i,j)$ denotes the distance of $i$ and $j$ in the graph $S_n$.
\end{lemma}
\begin{proof}
Let $M_i\subset \{1,\dots,n\}$ denote the set of all vertices $h$ such that $d(i,h) \leq 2$ and $X_h$ is within Euclidean distance $r/2$ of the center of the grid cell that contains $X_i$. The outline of the proof is the following: It suffices to show that $M_i$ contains a large constant times $\log n$ vertices. Since the same is true for $M_j$ and any two vertices in $M_i \cup M_j$ are within Euclidean distance $r$, with high probability there exists an edge between $M_i$ and $M_j$, establishing a path of length $5$ between $i$ and $j$. Let $N_i$ denote the set of $c$ neighbors picked by $i$. Then each $h \in N_i$ chooses its $c$ neighbors. Those that fall in the moon defined by $X_i$ and the center of the cell belong to $M_i$, see Figure \ref{fig:moon}.

\begin{figure}[ht]
\begin{center}
\leavevmode\includegraphics[scale=1.0,clip=false]{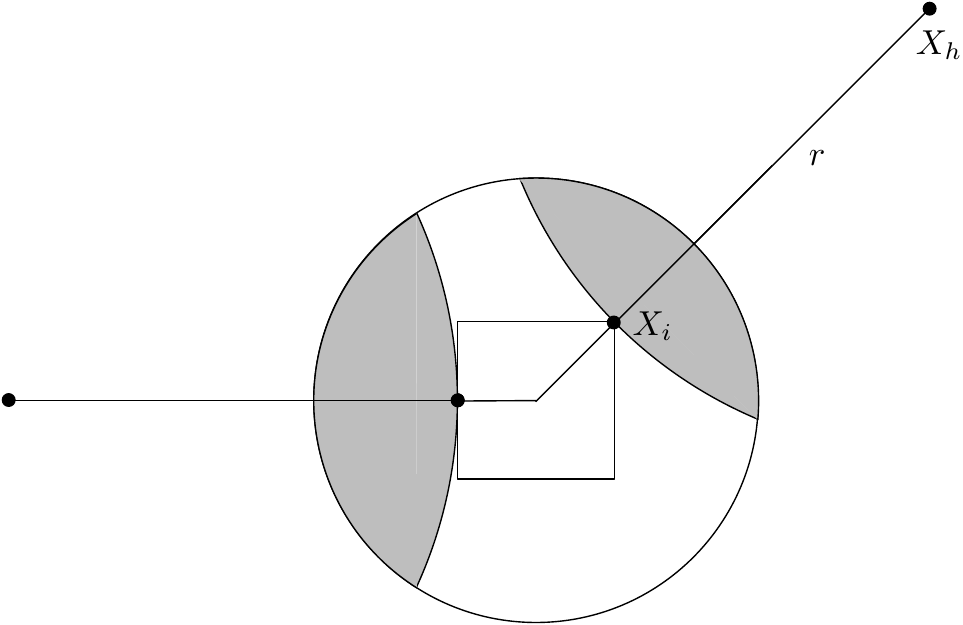}
\end{center}
\caption{\label{fig:moon}
For any point $X_i$ in a box of edge length $\ell$, if $h \in N_i$ is a neighbor selected by $i$, then $h$ may select neighbors within distance $r/2$ of the center of the square from the shaded regions--the so-called ``moons''.  The volume of any moon is at least a constant times $r^d$.
The figure shows two possible positions of $X_i$ in a box with a corresponding neighbor and moon.}
\end{figure}

Next we establish the required lower bound for the cardinality of $M_i$. Clearly, $|M_i|$ is at least as large as the number of neighbors selected by the vertices in $N_i$ that fall in $R$, the ball of radius $r/2$, centered at the mid-point of the cell into which $X_i$ falls. Denote by $h_1,\ldots,h_c$ the $c$ points belonging to $N_i$. Then 
\[
 | M_i | \geq | N_{h_1}\cap R | + | N_{h_2}\setminus N_{h_1} \cap R | + \cdots +
 | N_{h_c}\setminus \left(N_{h_1}\cup \ldots\cup N_{h_{c-1}}\right) \cap R |~.
\]
$h_1$ picks its $c$ neighbors among all points within distance $r$. The number of those neighbors falling in $R$ has a hypergeometric distribution. Since we are on $F$,  $| N_{h_1}\cap R |$ stochastically dominates $H_1$, a hypergeometric random variable with parameters $(c,\nu\log n,(\nu-\rho)\log n)$.
To lower bound the second term on the right-hand side, and to gain independence, remove all $c$ neighbors picked by $h_1$. Then $| N_{h_2}\setminus N_{h_1} \cap R |$ stochastically dominates $H_2$, a hypergeometric random variable with parameters $(c,\nu\log n-c,(\nu-\rho)\log n+c)$ (independent of $H_1$). Continuing this fashion, we obtain that $| M_i |$ is stochastically greater than $\sum_{i=1}^{c} H_i$ where the $H_i$ are independent and $H_i$ is hypergeometric with parameters $(c,\nu\log n-(i-1)c,(\nu-\rho)\log n+(i-1)c)$. Since $c \geq  \sqrt{(\rho/2)\log n}$, this may be bounded further as $\sum_{i=1}^{c} H_i$ is also stochastically greater than $\sum_{i=1}^{c} \wt{H}_i$ where the $\wt{H}_i$ are i.i.d.\ hypergeometric random variables with parameters $(c,(\nu/2)\log n,(3\nu/2)\log n)$.

Clearly, $\Exc{ \sum_{i=1}^{c} \wt{H}_i } = c^2/4  \geq (\nu/8)\log n$. We may bound the lower tail probabilities of $\sum_{i=1}^{c_n} \wt{H}_i$ by recalling an observation of \citet{Hoe63} according to which the expected value of any convex function of a hypergeometric random variable is dominated by that of the corresponding binomial random variable. Therefore, any tail bound obtained by Chernoff bounding for the binomial distribution also applies for the hypergeometric distribution. In particular,
\[
  \Pr{ \sum_{i=1}^{c} \wt{H}_i \leq \frac{1}{2}\Ex{ \sum_{i=1}^{c} \wt{H}_i } }
 \leq \exp{ \frac{-\Ex{ \sum_{i=1}^{c} \wt{H}_i }}{8} }
  \leq \exp{ - \frac{\nu}{64}\log n } \leq n^{-2}~.
\]
Thus, by the union bound, we obtain that
\[
  \Pr{ \exists i \in \{1,\dots,n\}: | M_i | \leq \frac{\nu\log n}{16} \;\Big\vert\; X_1,\dots,X_n, F} \leq \frac{1}{n}~.
\]
Thus, we have proved that with high probability, for every vertex $i$, the number of second generation neighbors (i.e., the neighbors selected by the neighbors selected by $i$) that end up within distance $r/2$ of the center of the grid cell containing $i$ is proportional to $\log n$. In particular, if $i$ and $j$ are two vertices in the same cell, then both $M_i$ and $M_j$ contain at least $(\nu/16)\log n$ points. If two of these points coincide, there is a path of length $4$ between $i$ and $j$. Otherwise, with very high probability, at least one vertex in $M_i$  selects a neighbor in $M_j$, creating a path of length five. Indeed, the probability that all neighbors selected by the vertices in $M_i$ miss all vertices in $M_j$, given that $| M_i |$ and $| M_j |$ are both greater than $(\nu/16)\log n$ and $M_i\cap M_j =\emptyset$ is at most
\begin{eqnarray*}
\prod_{h \in M_i}
\prod_{k=0}^{c-1}
\left(1-\frac{(\nu/16)\log n-k}{\rho\log n}\right)
\leq 
\left(1-\frac{\nu}{32\rho}\right)^{c (\nu/16)\log n}
\end{eqnarray*}
which goes to zero faster than any polynomial function of $n$. (Here we used that fact that $c \leq (\nu/32)\log n$ for a sufficiently large $n$.) Finally, we may use the union bound over all pairs of at most $\binom{n}{2}$ pairs of vertices $i$ and $j$ to complete the proof of the lemma.
\end{proof}

\begin{lemma}\label{lem:betweencell}
Suppose $X_1,\ldots,X_n$ are such that $F$ occurs. The probability that there exist two neighboring cells in the grid such that $S_n$ does not have any edge between the cells is bounded by $n^{-\delta}$ for all $\delta >0$ and sufficiently large $n$.
\end{lemma}
\begin{proof}
Since the cells are of edge length $\ell$, any two points in two neighboring squares are within distance $r$. Since under $F$ there are a logarithmic number of vertices in each cell, the same argument as at the end of the proof of Lemma \ref{lem:withincell} shows the statement.
\end{proof}

%



\bibliographystyle{abbrvnat}
\bibliography{bt}

\end{document}